\documentclass[12pt,a4paper]{amsart} %jesli chce numeracje jednostronnq, to daje ,oneside, w graniaste.
\usepackage{amsmath}
\usepackage{amsthm}
\usepackage{amssymb}
\usepackage{mathrsfs}%to reczne ladne pismo \mathscr

\newtheorem{thm}{Theorem}[section]
\newtheorem{lem}[thm]{Lemma}

\newtheorem{prop}[thm]{Proposition}

\theoremstyle{remark}
\newtheorem{rem}[thm]{Remark}
\newtheorem*{rem*}{Remark}

\theoremstyle{definition}

\newtheorem{ex}[thm]{Example}

\numberwithin{equation}{section}

\newcommand{\om}{\Omega}

\newcommand{\C}{\mathbb{C}}%{\bf C}
\newcommand{\Rz}{\mathbb{R}}

\begin{document}
\title{The Bernstein-Walsh-Siciak Theorem for analytic hypersurfaces}

\author{Anna Denkowska \& Maciej P. Denkowski}%\address{Jagiellonian University, Faculty of Mathematics and Computer Science, \L ojasiewicza 6. 30-348 Krak\'ow, Poland}\email{maciej.denkowski@uj.edu.pl}

\date{February 28th 2014, last revised: December 30th 2018}
\keywords{Complex analytic and algebraic sets, Kuratowski convergence of sets, approximation theory, Bernstein-Walsh-Siciak Theorem, analytic multifunctions.}
\subjclass{32B15, 32E30}

\begin{abstract}
As a first step towards a general set-theoretic counterpart of the remarkable Bernstein-Walsh-Siciak Theorem concerning the rapidity of polynomial approximation of a holomorphic function on polynomially convex compact sets in ${\C}^n$, we prove a version of this theorem for analytic hypersurfaces.
\end{abstract}

\maketitle

\centerline{\it In memoriam Professor J\'ozef Siciak}

%\bigskip
%\noindent{\bf\large 1. 
\section{Preliminaries}
The classical Bernstein-Walsh-Siciak Theorem, together with its converse, is the following result (due to Bernstein in the case of a real segment in ${\C}$, to Walsh for any compact subset of the complex plane and to Siciak in the general setting, see \cite{S}):
\begin{thm}\label{BWS}
Let $K\subset{\C}^m$ be a nonempty, polynomially convex compact set and $f\colon K\to {\C}$ a continuous function. Then \begin{enumerate}
\item if $f$ is the restriction to $K$ of a holomorphic function defined in a neighbourhood of $K$ in ${\C}^m$, then 
$$
\limsup_{d\to+\infty}\sqrt[d]{\mathrm{dist}_K(f,\mathcal{P}_d({\C}^m))}<1;\leqno{(\#)}
$$
\item if  $(\#)$ holds and in addition the Siciak extremal function $\Phi_K$ is continuous, then $f$ is the restriction to $K$ of a holomorphic function defined in a neighbourhood of $K$ in the ambient space.
\end{enumerate}
\end{thm}
Here we denote by $\mathcal{P}_d({\C}^m)$ the space of complex polynomials in $m$ variables, of degrees $\leq d$. For any compact nonempty set $K\subset{\C}^m$ and a continuous function $f\colon K\to {\C}$ we put
$$
\mathrm{dist}_K(f,\mathcal{P}_d({\C}^m)):=\inf\{||f-P||_K\mid P\in \mathcal{P}_d({\C}^m)\},
$$
where $||f||_K:=\sup_{x\in K}|f(x)|$ is the usual {\it Chebyshev norm}. For the convenience of the reader we recall briefly that a nonempty compact set $K\subset{\C}^m$ is said to be {\it polynomially convex} iff it coincides with its {\it polynomial hull} $\widehat{K}:=\{x\in{\C}^m\mid |P(x)|\leq ||P||_K, P\in\mathcal{P}({\C}^m)\}$ where $\mathcal{P}({\C})^m$ is the space of all complex polynomials in $m$ variables, and the {\it Siciak extremal function} of $K$ is defined to be $\Phi_K(x)=\sup\{|P(x)|^{1/\deg P}\mid P\in\mathcal{P}({\C}^m)\setminus{\C}\colon ||P||_K\leq 1\}$.

In other words, the Bernstein-Walsh-Siciak Theorem establishes an equivalence between holomorphicity and the geometric rate of polynomial approximation (recall that the possibility of such an approximation on polynomially convex compact sets is due to the famous Oka-Weil Theorem). It plays an important role in approximation theory and still inspires research, see e.g. \cite{P}, \cite{P2}.

Recall that holomorphic functions are characterized among continuous functions by the analycity of their graphs. Similarly, a continuous function defined on the whole of ${\C}^m$ is a polynomial iff its graph is algebraic, by the Serre Theorem. 

A natural question that arises in complex analytic geometry is whether a Bernstein-Walsh-Siciak-type result is true for the approximation of analytic sets by algebraic ones in the sense of the Kuratowski convergence which is a natural convergence of closed sets generalizing the convergence defined by the Hausdorff distance (in the complex analytic case it is closely related to the convergence of integration currents on analytic sets). It is easy to check that a sequence of continuous functions is convergent locally uniformly iff their graphs converge in the sense of Kuratowski (and the limit function is continuous). A particular motivation comes from the recent beautiful algebraic approximation results of Bilski, e.g. \cite{B1}, \cite{B2}. The first step towards a set-theoretic version of the Bernstein-Walsh-Siciak Theorem is given in the next section, Theorems \ref{BWS hyper} and \ref{converse} deal with the case of hypersurfaces.

Let us briefly recall the notion of the Kuratowski convergence. Let $\om\subset{\C}^m$ be a locally closed, nonempty set. We denote by $\mathcal{F}_\om$ pothe family of all its closed subsets. This space is endowed with a natural, metrizable convergence, called {\it Painlev\'e-Kuratowski or Kuratowski  convergence}, which turns it into a {\sl compact metric} space (see e.g. \cite{TW1}, the compactness is an old result of Zarankiewicz, more details can be found in \cite{RW}). Namely, if $F_\nu,F\in\mathcal{F}_\om$, then we will write $F_\nu\stackrel{K}{\longrightarrow} F$ if and only if\begin{enumerate}
\item Any point $x\in F$ is the limit of some sequence of points $x_\nu\in F_\nu$;
\item For all compact $K\subset \om\setminus F$ there is $K\cap F_\nu=\varnothing$, from some index $\nu_0$ onwards.
\end{enumerate}
If $\om$ is compact, then this convergence is given by the extended {\it Hausdorff metric} defined by 
$$
d_H(E,F):=\begin{cases} \max\{\max_{x\in E}\mathrm{dist}(x,F),\max_{x\in F}\mathrm{dist}(x,E)\},&\hbox{if}\ E,F\neq\varnothing;\\
0,&\hbox{if}\ E=F=\varnothing;\\
\mathrm{diam}\om+1,&\hbox{otherwise.}
\end{cases}
$$
To fix the attention we may assume that all the distances are computed in the usual Euclidean norm. It is easy to see that if $E$ and $F$ are both nonempty, then 
$$
d_H(E,F)=\inf\{r>0\mid E\subset F+\overline{B_m(r)}\quad\hbox{and}\quad F\subset E+\overline{B_m(r)}\},
$$
where $B_m(r)$ is the open  Euclidean ball of radius $r$, centred at the origin. We will denote by $P_m(r)=B_1(r_1)\times\ldots\times B_1(r_m)$ the polydisc centred at $0\in{\C}^m$ and with multiradius $r=(r_1,\dots, r_m)$. For a nonempty subset $\mathcal{W}\subset\mathcal{F}_\Omega$ and $E\in \mathcal{F}_\Omega$, we write $d_H(E,\mathcal{W})=\inf\{d_H(E,W)\mid W\in \mathcal{W}\}$.

%%%%%%%%%%%%%%%%%%%%%%%%%55555

For more informations see \cite{TW1}, \cite{TW}.

Finally, recall the H\"older continuity property of roots in the version of \cite{St} (Theorem 2.1.2).
\begin{prop}\label{Hoelder}
For any $a=(a_1,\dots, a_n)\in{\C}^n$ let $P_a(t)=t^n+a_1t^{n-1}+\ldots+a_n$ have $\zeta^a_1,\dots,\zeta^a_n$ as all roots (counted with multiplicities). Let $|\cdot|$ be the maximum norm in ${\C}^n$ and suppose that $C>1$ is such that $|a|\leq C$. Then for any $b\in{\C}^n$ with $|b|\leq C$ one can renumber the roots of $P_b$ in such a way that 
$$
\forall j\in\{1,\dots, n\},\quad |\zeta^a_j-\zeta^b_j|\leq 4nC|a-b|^{\frac{1}{n}}.
$$
\end{prop}

\begin{rem}\label{rwz}
It is useful to observe that the condition 
$$
\limsup_{n\to+\infty} \sqrt[n]{\alpha_n}<1
$$ where $\alpha_n$ is a sequence of non-negative real numbers, is equivalent to the existence of two constants $M>0,\theta\in (0,1)$ such that
$$
\alpha_n\leq M\theta^n,\quad n=0,1,2,\ldots
$$
\end{rem}

\section{The case of hypersurfaces}
In this section we will prove a theorem of Bernstein-Walsh-Siciak type (and its converse) for analytic multifunctions as N.V. Shcherbina \cite{Sh} calls them. The special role played in geometry by the `multigraphs' we are considering here may be inferred from \cite{B2}, \cite{Sh}, \cite{TW3} and \cite{Ch}, \cite{L}.

\subsection{The setting}

Let $\pi(x,t)=x$ for $(x,t)\in{\C}^m\times{\C}$ and let $K\subset{\C}^m$ be a nonempty compact set. Basically, we will be dealing with sets $Y\subset K\times {\C}\subset{\C}^m\times {\C}$ such that $\pi(Y)=K$ and the $x$-sections of $Y$ are compact. The first step towards a general set-theoretic counter-part of the Benrstein-Walsh-Siciak Theorem is to consider sets $Y$ that can be described as the zero-set of a single polynomial in $t$ with coefficients that are continous in $x$. 

Since it is most convenient to treat such sets  $Y$ as finite, continuous (in the sense of the Hausdorff metric cf. Proposition \ref{Hoelder}) multifunctions $K\to \mathscr{P}(\C)$, we refer the reader to \cite{RW} for details (chapters 4 and 5). Instead of directly using the Hausdorff metric, we will adopt a slightly different point of view, much closer in some sense to the classical Bernstein-Walsh-Siciak Theorem \ref{BWS}. %This is motivated by Example \ref{recenzent}.

We denote by $Y(x)\subset{\C}$ the value of the multifunction $Y$ at $x\in K$, i.e. the section of the set $Y$ at the point $x$. Given a nonempty family $\mathcal{H}_K$ of compact-valued, continuous multifunctions $K\to \mathcal{P}({\C})$, for any $W\in \mathcal{H}_K$ we may consider
$$
\delta_K(Y,W):=\sup_{x\in K}d_H(Y(x),W(x))
$$
as a natural counterpart of the Chebyshev norm $||f-P||_K$. Note that if we restrict ourselves e.g. to the space of all multifunctions $K\to\mathcal{P}({\C})$ with compact graphs, $\delta_K$ is indeed a metric. Moreover, it is easy to see that in this situation $$d_H(Y,W)\leq \delta_K(Y,W).$$

Next, we put
$$
\mathrm{dist}_\delta(Y,\mathcal{H}_K):=\inf\{\delta_K(Y,W)\mid W\in \mathcal{H}_K\}.
$$

The main question is now what should we exactly replace $\mathcal{P}_d({\C}^m)$ with.

\subsection{A generalization of Theorem \ref{BWS} (1)}

Consider a complex analytic set $X\subset D\times {\C}$ of pure codimension 1 (i.e. a hypersurface) with proper projection $\pi(x,t)=x$ onto the domain $D\subset{\C}^m$. Note that every analytic hypersurface admits locally a coordinate system in which it satisfies these assumptions. For a nonempty set $K\subset D$ we put $X_K:=X\cap (K\times{\C})$. 

For a given algebraic hypersurface $V\subset {\C}^{m+1}$ we can find a reduced polynomial $P$ (unique up to a unit) such that $V=P^{-1}(0)$. Then the projective degree $\deg V$ coincides with $\deg P$. If we assume, moreover, that $\pi|_V$ is proper, then we can choose $P$ of the form  $P(x,t)=t^d+a_1(x)t^{d-1}+\ldots+a_d(x)$ with polynomial coefficients (and the leading coefficient is identically equal to 1) and all the degrees $\deg a_j$ and $d$ do not exceed $\deg V$. Note that any algebraic subset of ${\C}^{m+1}_{x,t}$ described by a polynomial that is monic in $t$ has proper projection onto the first $m$ coordinates. We refer the reader to \cite{L} and \cite{Ch} for details.

Consider now two families of sets:
\begin{align*}
\mathcal{V}_d^m=\{V\subset{\C}^{m+1}\mid V\ \textrm{is algebraic of pure dimension}\ m,\ \deg V\leq d\}\cup\{\varnothing\},
\end{align*}
together with 
\begin{align*}
\mathcal{H}_d^m(K):=\{V_K\mid & V\in\mathcal{V}_d^m\setminus\{\varnothing\}\colon V\ \textrm{has proper projection onto}\ {\C}^m\}\cup\{\varnothing\}.
\end{align*}

We also put $\mathcal{V}_d^m(K):=\{V_K\mid V\in \mathcal{V}_d^m\}$.

\begin{rem}
By the main result of \cite{TW}, $\mathcal{V}_d^m$, %and hence $\mathcal{V}_d^m(K)$, 
is closed in the topology of the Kuratowski convergence. However, it is not the case for $\mathcal{H}_d^m(K)$. 

To see the latter consider the sequence of graphs $$W_\nu:=\{(x,\nu(x^2-(1/4)))\mid |x|\leq 1\}\in\mathcal{H}_2^2(\overline{B_1(1)});$$ it converges to $W=\{-1/2,1/2\}\times{\C}\notin\mathcal{H}_2^2(\overline{B_1(1)})$.
%
%In order to prove that $\mathcal{V}_d^m(K)$ is closed, we take any convergent sequence of sets $\mathcal{H}_d^m(K)\ni W_\nu\stackrel{K}{\longrightarrow} W\in\mathcal{F}_{{\C}^{m+1}}$. We may assume that $W\neq\varnothing$ and consider the corresponding algebraic sets $V^{(\nu)}\in\mathcal{V}_d$ such that $V_K^{(\nu)}=W_\nu$. We choose these sets to be the smallest possible (in the sense of inclusion; this does not necessarily mean unicity). Then, extracting a subsequence, if necessary, we may assume that $V^{(\nu)}\stackrel{K}{\longrightarrow} V$ where $V\in\mathcal{V}_d^m$ by \cite{TW}. Clearly, $V_K\supset W$ (cf. (1) in the definition of the convergence). 
\end{rem}

Now we are ready to prove our first main result:
\begin{thm}\label{BWS hyper}
In the setting introduced above assume that $K\Subset D$ is a compact, poynomially convex set. Then 
$$
\limsup_{d\to+\infty}\sqrt[d]{\mathrm{dist}_\delta(X_K,\mathcal{V}_d^m(K))}<1.\leqno{(\#\#)}
$$
In particular, if $X=F^{-1}(0)$ where $F(x,t)=t^n+a_1(x)t^{n-1}+\ldots+a_n(x)\in\mathcal{O}(D)[t]$, then there are constants $M>0$, $\theta\in(0,1)$ and polynomials $p_d(x,t)=t^n+a_{1,d}(x)t^{n-1}+\ldots+a_{n,d}(x)$ in $m+1$ variables such that $\deg p_d^{-1}(0)\leq d$, the polynomial coefficients $a_{j,d}$ converge uniformly to $a_j$ on $K$ and 
$$
\delta_K(X_K, p_d^{-1}(0)_K)\leq M\theta^d, \quad d=0,1,2,\dots
$$
\end{thm}
\begin{proof}
Note that $X_K\neq\varnothing$.

First, we use the classical Andreotti-Stoll results to describe $X$ as the set of zeroes of a reduced pseudopolynomial (an optimal polynomial, i.e. with discriminant non identically equal to zero in $D$) $F(x,t)=t^n+a_1(x)t^{n-1}+\ldots+a_n(x)\in\mathcal{O}(D)[t]$. 

Next, we can apply the classical Bernstein-Walsh-Siciak Theorem in order to find polynomials $a_{j,d}\in\mathcal{P}_d({\C}^m)$ such that for all $d\in\mathbb{N}$,
$$
||a_j-a_{j,d}||_K\leq M\theta^d,\quad j=1,\dots,n,
$$
where the constants $M>0$ and $\theta\in(0,1)$ are chosen independent of $j$. Now, put
$$
P_d(x,t):=t^n+a_{1,d}(x)t^{n-1}+\ldots+a_{n,d}(x)
$$
and observe that $\deg P_{d}^{-1}(0)\leq 2d-1$ for $d\geq n$. Indeed, as observed earlier, $\deg P_d^{-1}(0)\leq \deg P_d\leq \max\{n, \deg a_{1,d}+n-1,\ldots, \deg a_{n,d}\}$. Then for $V^{(d)}:=P_d^{-1}(0)$, we have certainly $V_K^{(d)}\in \mathcal{H}_{2d-1}^m(K)\subset\mathcal{H}_{2d}^{m}(K)$. 

Fix $x\in K$ and write $t_1(x),\dots,t_n(x)$ for the roots of $F(x,\cdot)$ counted with multiplicities. The H\"older continuity of roots (Proposition \ref{Hoelder}) ensures us that after a suitable renumbering of the roots $t_{j,n}(x)$ of $P_d(x,\cdot)$ we have
$$
|t_j(x)-t_{j,d}(x)|\leq 4nC\max_{i=1}^n|a_i(x)-a_{i,d}(x)|^{1/n}\leq 4nCM^{1/n}\theta^{d/n}
$$
where $C$ is an appropriate constant. For instance, $C=\max_{i=1}^n||a_i||_K+M$ is good for all $d$ large enough, since $||a_{i,d}||_K\leq ||a_{i,d}-a_i||_K+||a_i||_K$. It follows now that there is a constant $\tilde{M}>0$ such that
$$
d_H(X(x),V^{(d)}(x))\leq \tilde{M}\theta^{d/n},\quad x\in K,
$$
and so for $\tilde{\theta}:=\theta^{1/(2n)}\in (0,1)$, we obtain eventually
$$
\mathrm{dist}_\delta(X_K,\mathcal{V}_{2d}^m(K))\leq \mathrm{dist}_\delta(X_K,\mathcal{H}_{2d}^m(K))\leq \delta_K(X_K,V_K^{(d)})\leq \tilde{M}\tilde{\theta}^{2d}.
$$
On the other hand, we have seen above that $V_K^{(d)}\in\mathcal{H}_{2d-1}^m(K)$, too. Therefore, since $\tilde{\theta}\in(0,1)$, we obtain directly
$$
\mathrm{dist}_\delta(X_K,\mathcal{V}_{2d-1}^m(K))\leq \mathrm{dist}_\delta(X_K,\mathcal{H}_{2d-1}^m(K))\leq \delta_K(X_K,V_K^{(d)})\leq\tilde{M}\tilde{\theta}^{2d}\leq \tilde{M}\tilde{\theta}^{2d-1}.
$$
%%%%%%5
%
% starting again from
%$$
%\delta_K(X_K,V^{(d)})\leq \tilde{M}\theta^{d/n},
%$$
%and writing $\theta^{d/n}=(\theta^{d/[n(2d-1)]})^{2d-1}$, in view of the fact that $\theta^{1/n}\in(0,1)$ and $d/(2d-1)\geq 1/2$, we get %$\theta^{d/n}\leq (\theta^{1/(2n)})^{2d-1}$. Hence, for $\tilde{\theta}=\theta^{1/(2n)}$, as before, we obtain
%$$
%\mathrm{dist}_\delta(X_K,\mathcal{V}_{2d-1}^m(K))\leq \tilde{M}\tilde{\theta}^{2d-1}.
%$$
and we are done (cf. Remark \ref{rwz}).
\end{proof}
\begin{rem}\label{uwaga}
In the course of the proof we obtain actually $$\limsup_{d\to+\infty}\sqrt[d]{\mathrm{dist}_\delta(X_K,\mathcal{H}_d^m(K))}<1.$$
\end{rem}
\begin{rem}\label{uwaga 2}
Since in our setting there is $d_H\leq \delta_K$, the Theorem above holds true also for the Hausdorff metric, i.e.
$$
\limsup_{d\to+\infty}\sqrt[d]{d_H(X_K,\mathcal{H}_d^m(K))}<1.
$$
\end{rem}

\subsection{A generalization of Theorem \ref{BWS} (2)}

In order to prove a converse to the last Theorem, let us consider now the following setting. Put
$$
Y:=\{(x,t)\in K\times{\C}\mid a_0(x)t^n+a_1(x)t^{n-1}+\ldots+a_n(x)=0\}
$$
where $\varnothing\neq K\subset{\C}^m$ is compact, $a_j\colon K\to {\C}$ are continuous and $a_0\not\equiv 0$. It is easy to see that assuming  $\pi(x,t)=x$ is proper on $Y$ is equivalent to say that $a_0^{-1}(0)=\varnothing$ (compare \cite{TW3}). Thus, we may as well assume that $a_0\equiv 1$. We will also assume that for some $x_0\in K$, there are exactly $n$ different points in the fibre $\pi^{-1}(x_0)\cap Y$. Let $F(x,t)=t^n+a_1(x)t^{n-1}+\ldots+a_n(x)$ be the defining function of $Y$.

We introduce a new family of sets:
\begin{align*}
\mathcal{H}_d^{m,n}(K):=\{V_K\mid &V\in\mathcal{V}_d^m\setminus\{\varnothing\}\colon V\ \textrm{has proper projection onto}\ {\C}^m\\
&\textrm{with covering number}\ \leq n\}.
\end{align*}
Necessarily, $n\leq d$ (cf. \cite{L}, see also \cite{TW3}).

\begin{rem}
It is obvious from the proof of Theorem \ref{BWS hyper} that we are actually dealing there with sets $V^{(d)}_K\in\mathcal{H}_d^{m,n}(K)$ so that -- in view of Remark \ref{uwaga} -- we can replace $(\#\#)$ with 
$$
\limsup_{d\to+\infty}\sqrt[d]{\mathrm{dist}_\delta(X_K,\mathcal{H}_d^{m,n}(K))}<1.\leqno{(\#\#\#)}
$$
\end{rem}

We can now prove a true converse to Theorem \ref{BWS hyper}.

\begin{thm}\label{converse}
In the setting introduced above, assume moreover that $K$ is polynomially convex and $\Phi_K$ is continuous. Then the condition 
$$
\limsup_{d\to+\infty}\sqrt[d]{\mathrm{dist}_\delta(Y,\mathcal{H}_d^{m,n}(K))}<1,
$$
implies that for some neighbourhood $U\supset K$ there is an analytic set $X\subset U\times{\C}$ of pure codimension $1$, having proper projection onto $U$ and such that $X_K=Y$.

In particular, as $Y=F^{-1}(0)$ with $F(x,t)=t^n+a_1(x)+\ldots+a_n(x)\in\mathcal{C}(K,{\C})[t]$, each coefficient $a_j$ admits a holomorphic extension $\tilde{a}_j\in \mathcal{O}(U)$ and $X=\tilde{F}^{-1}(0)$ for $\tilde{F}(x,t)=t^n+\tilde{a}_1(x)t^{n-1}+\ldots+\tilde{a}_n(x)\in\mathcal{O}(U)[t]$.
\end{thm}

Before the proof let us note the following basic lemma.
\begin{lem}\label{rachunek}
Let $n\geq 1$, $r>0$ and take $t_1,\dots,t_n,s_1,\dots, s_n\in{\C}$. Fix any $R\geq\max_{i=1}^n|t_i|$ and assume that $\max_{i=1}^n|t_i-s_i|\leq r$. Then for any $k\in\{1,\dots,n\}$ and any set of indices $1\leq i_1<\dots<i_k\leq n$, there is
$$
|t_1\cdot\ldots\cdot t_k-s_1\cdot\ldots\cdot s_k|\leq C_k r
$$
where $C_1=1$ and $C_k=R^{k-1}+(r+R)C_{k-1}$ with $C_{k-1}$ denoting the constant for $k-1$.
\end{lem}
\begin{proof} We proceed by induction on $k$. Clearly, for $k=1$ we can take $C_1=1$. Fix $k>1$ and suppose we have the required inequality for $k-1$ with the constant $C_{k-1}$. Write
\begin{align*}
|&t_{i_1}\cdot\ldots\cdot t_{i_k}-s_{i_1}\cdot\ldots\cdot s_{i_k}|=|t_{i_1}\cdot\ldots\cdot t_{i_k}-s_{i_1}\cdot\ldots\cdot s_{i_k}\pm s_{i_1}\cdot t_{i_2}\cdot\ldots\cdot t_{i_k}|\leq \\
&\leq |t_{i_1}-s_{i_1}|\cdot|t_{i_2}\cdot\ldots\cdot t_{i_k}|+
(|s_{i_1}-t_{i_1}|+|t_{i_1}|)|t_{i_2}\cdot\ldots\cdot t_{i_k}-s_{i_2}\cdot\ldots\cdot s_{i_k}|\leq\\
&\leq rR^{k-1}+(r+R)C_{k-1}r=C_kr
\end{align*}
with $C_k:=R^{k-1}+(r+R)C_{k-1}$. 
\end{proof}

\begin{proof}[Proof of Theorem \ref{converse}]
There are constants $M>0, \theta\in (0,1)$ such that for each $d\in \mathbb{N}$ we can find $W_d\in \mathcal{H}_d^{m,n}(K)$ satisfying
$$
\delta_K(Y,W_d)\leq M\theta^d. \leqno{(\ast)}
$$
In particular, $W_d\stackrel{K}{\longrightarrow}Y$. 

It follows from the definition of $\mathcal{H}_d^{m,n}(K)$ that for each $d$, there are polynomials $a_{j,d}$, $j=1,\dots,n_d\leq n$ such that $W_d=P_d^{-1}(0)_K$, where $P_d(x,t)=t^{n_d}+a_{1,d}(x)t^{n_d-1}+\ldots+a_{n_d,d}(x)$. Then $\deg W_d\leq d$ implies $\deg a_{j,d}\leq d-n+j$ and $n\leq d$ (so that $\deg a_{j,d}\leq 2d-1$).

Take $x_0\in K$ for which there are $n$ pairwise different points $(x_0,t_i)\in Y$, $i=1,\dots, n$. Separate the points $t_i$ by pairwise disjoint closed discs $D_i:=t_i+\overline{B_1(r)}$. Then it is straightforward from the form of $Y$ that there is a bounded (arbitrarily small), connected neighbourhood $G$ of $x_0$ such that $(\overline{G}\times\bigcup_{i=1}^n \partial D_i)\cap Y=\varnothing$. At the same time $({G}\times\mathrm{int} D_i)\cap Y\neq \varnothing$, for each $i$. Therefore, by the convergence $W_d\stackrel{K}{\longrightarrow} Y$, we easily conclude that $P_d(x_0,\cdot)$ has at least $n$ different roots, whenever $d$ is large enough, i.e. $n_d\geq n$, $d\gg 1$. Indeed, for $d\gg 1$ and each $i$, $(\overline{G}\times\partial D_i)\cap W_d=\varnothing$, while $(G\times\mathrm{int} D_i)\cap W_d\neq\varnothing$ which means that $(G\times D_i)\cap W_d\neq\varnothing$ has proper projection on $G$ and thus by the Remmert Theorem $\pi((G\times D_i)\cap W_d)=G$, whence $(\{x_0\}\times D_i)\cap W_d\neq\varnothing$. Summing up, we may assume that $n_d=n$, for all $d$. 

The next step consists in observing that each $a_{k,d}$ converges uniformly to $a_k$ on $K$ and the rate of convergence is geometric. Indeed, $(*)$ implies that for each $x\in K$, renumbering the roots $t_1(x),\dots, t_n(x)$ of $F(x,\cdot)$ and $t^{(d)}_1(x),\dots,t^{(d)}_n(x)$ of $P_d(x,\cdot)$ adequately, we have $|t_i(x)-t_i^{(d)}(x)|\leq M\theta^d$, for any $x\in K$. On the other hand, since the coefficients $a_k(x)$ and $a_{k,d}(x)$ are expressed in terms of the elementary symmetric polynomials of the roots, a careful application of Lemma \ref{rachunek} (\footnote{Note that the roots $t_i(x)$ are uniformly bounded from above and it is enough to estimate the differences $|t_{i_1}(x)\cdot\ldots\cdot t_{i_k}(x)-t^{(d)}_{i_1}(x)\cdot\ldots\cdot t^{(d)}_{i_k}(x)|$ uniformly w.r.t. $x\in K$, for any $1\leq i_1<\ldots<i_k\leq n$, since the $k$-th coefficient $a_k(x)$ is equal to $(-1)^k\sum_{1\leq i_1<\ldots<i_k\leq n}t_{i_1}(x)\cdot\ldots\cdot t_{i_k}(x)$. Lemma \ref{rachunek} for $r=M\theta^d$  and any $R>0$ such that $Y\subset K\times B_1(R)$ yields $||a_k-a_{k,d}||_K\leq D_kM\theta^d$ with $D_k=\binom{n}{k}(R^{k-1}+(M+R)C_{k-1})$ for $k=2,\dots, n$ and $D_1=n$ (since $\theta^d<1$ we can replace $M\theta^d$ in $D_k$ and in all the $C_j$'s, $j<k$, with $M$).}) shows that there is a constant $\tilde{M}>0$ such that
$$
||a_k-a_{k,d}||_K\leq \tilde{M}{\theta}^d,\quad k=1,\dots, n;\> d\in\mathbb{N}. 
$$
It remains to observe that taking $\tilde{\theta}:=\sqrt{\theta}\in (0,1)$ together with the fact that $\tilde{\theta}^{2d}\leq \tilde{\theta}^{2d-1}$ yields eventually
$$
\mathrm{dist}_K(a_j,\mathcal{P}_k({\C}^m))\leq\tilde{M}\tilde{\theta}^k,\quad j=1,\dots,n.
$$
Applying the classical converse to the Bernstein-Walsh-Siciak Theorem, we obtain holomorphic extensions $a_i\subset \tilde{a}_i$ onto a common neighbourhood $U$ of $K$, which implies that $Y=X_K$ for $X=\tilde{F}^{-1}(0)$ where $\tilde{F}(x,t)=t^n+\tilde{a}_1(x) t^{n-1}+\ldots+\tilde{a}_n(x)\in\mathcal{O}(U)[t]$. This is the assertion sought after.
\end{proof}

The situation here is not exactly as before and we do not have the conclusion from Remark \ref{uwaga 2}. As a matter of fact there is no easy transition from   $\delta_K$ to $d_H$ in the last Theorem. This is illustrated by the following Example for which we thank warmly an anonymous reader.
\begin{ex}\label{recenzent}
Consider the sequences $a_0=0$ and $a_k=a_{k-1}+(1/2^k)$ for $k\geq 1$ on the one hand, while on the other, $b_0=0$, $b_k=b_{k-1}+(1/k^2)$, $k\geq 1$. Then $a_k\to 1$, whereas $b_k\to\pi^2/6$. Define $f\colon [0,1]\to[0,\pi^2/6]$ to be the function whose graph is obtained from the segments $[(a_k,b_k),(a_{k+1},b_{k+1})]\subset{\Rz}^2$ with $f(1)=\pi^2/6$ for continuity. 

Next, for $k\geq 1$, let $f_k$ be equal to $f$ everywhere on $[0,1]\setminus[a_k,a_{k+1}]$, while on $[a_k,a_{k+1}]$ we define the graph of $f_k$ to be the segment joining the points $(a_k,b_k)$ and $((a_k+a_{k+1})/2,b_{k+1})$ over $[a_k,(a_k+a_{k+1}/2)]$, and $f_k(x)=b_{k+1}$, whenever $x\in [(a_k+a_{k+1})/2,a_{k+1}]$. 

Then it is apparent that the Hausdorff distance between the graphs of $f$ and $f_k$ is at most $1/2^k$, while $||f-f_k||_{[0,1]}=1/(2k^2)$. Therefore, even though the rate of convergence of the graphs of $f_k$ to the graph of $f$ in the Hausdorff distance is geometric, the rate of uniform convergence of $f_k$ to $f$ is not (\footnote{This is essentially due to the fact that computing the Hausdorff distance involves `looking in various directions', while computing the Chebyshev norm restricts to a `vertical point of view'.}).

As a consequence, if we try to prove Theorem \ref{converse} along the same lines as before, but under the assumption 
$$
\limsup_{d\to+\infty}\sqrt[d]{d_H(Y,\mathcal{H}_d^{m,n}(K))}<1,
$$
there seems to be no simple way of concluding that the rate of convergence of $||a_j-a_{j,d}||_K$ to zero is geometric as we would like it to be. The point is that repeating the argument from the proof of Theorem \ref{converse}, we see that now $d_H(Y, W_d)\leq M\theta^d$ (that replaces $(*)$) implies that
after a suitable renumbering of the roots $t_1(x),\dots, t_n(x)$ of $F(x,\cdot)$ and $t^{(d)}_1(x),\dots,t^{(d)}_n(x)$ of $P_d(x,\cdot)$, we have $|t_i(x)-t_i^{(d)}(x)|\leq C_dM\theta^d$, for some constant $C_d>0$ independent of $x$ (\footnote{This is a general statement about multifunctions: if $F, G\colon K\to{\Rz}^k$ are compact, continuous multifunctions on a compact $K\subset{\Rz}^m$ and $d_H(\mathrm{graph}(F),\mathrm{graph}(G))\leq r$, then for some $C>0$, $d_H(F(x),G(x))\leq Cr$, for all $x\in K$. Otherwise, we would find a sequence $n_1<n_2<\dots$ in $\mathbb{N}$ and points $x_{n_\ell}\in K$ convergent to some $x_0\in K$ and such that $d_H(F(x_{\nu_\ell}),G(x_{\nu_\ell}))\geq n_\ell r$, for all $\ell=1,2,\dots$, which yields a contradiction, since the left-hand side converges to $d_H(F(x_0),G(x_0))$.}), but dependent on $d$. Therefore, we are lead to the inequality $||a_j-a_{j,d}||\leq C_dM\theta^d$ and we do not know whether it is possible to replace $C_dM$ by a constant $\tilde{M}$ independent of $d$.
\end{ex}

\section{Acknowledgements}
This work was begun during the second author's stay at the University Lille 1; 
it was partially supported by  Polish Ministry of Science
and Higher Education grant 1095/MOB/2013/0. 

The authors would like to thank professor W. Ple\'sniak for the introduction into the subject and professor T. Winiarski for suggesting the problem several years ago.

Last but not least, the authors are deeply indebted to a voluntarily anonymous reader for spotting a flaw in the first version of Theorem \ref{converse} and most grateful for suggesting Example \ref{recenzent}.

\bigskip
\noindent{\small   Cracow University of Economics\hfill Jagiellonian University\\
Department of Mathematics\hfill Faculty of Mathematics and Computer Science\\
Rakowicka 27\hfill \L ojasiewicza 6\\
31-510 Krak\'ow, Poland\hfill 30-348 Krak\'ow, Poland\\
e-mail addresses: {\tt anna.denkowska@uek.krakow.pl}\hfill {\tt denkowsk@im.uj.edu.pl}}

\end{document}